\documentclass[12pt,reqno]{amsart}

\usepackage{graphicx}
\graphicspath{ {./images/} }
\usepackage{bm}

\usepackage{color}

\usepackage{amssymb}

\usepackage{amsfonts}

\usepackage{amsmath}

\usepackage[all]{xy}

\newtheorem{theorem}{Theorem}[section]

\newtheorem{corollary}[theorem]{Corollary}

\newtheorem{lemma}[theorem]{Lemma}

\newtheorem{proposition}[theorem]{Proposition}

\newtheorem{Definition}[theorem]{Definition}

\newtheorem{Example}[theorem]{Example}

\newtheorem{Remark}[theorem]{Remark}

\newenvironment{remark}{\begin{Remark}\begin{em}}{\end{em}\end{Remark}}

\DeclareMathOperator{\tr}{tr}

\address{Miran Jeong \\ Department of Mathematics, Chungbuk National University, Cheongju 28644, Korea}
\email{jmr4006@chungbuk.ac.kr}

\address{Sejong Kim \\ Department of Mathematics, Chungbuk National University, Cheongju 28644, Korea}
\email{skim@chungbuk.ac.kr}

\begin{document}

\title[Weak log-majorization and inequalities of power means]{Weak log-majorization and inequalities of power means}

\author{Miran Jeong and Sejong Kim}

\maketitle

\begin{abstract}
As non-commutative versions of the quasi-arithmetic mean, we consider the Lim-P\'{a}lfia's power mean, R\'{e}nyi right mean and  R\'{e}nyi power means.
We prove that the Lim-P\'{a}lfia's power mean of order $t \in [-1,0)$ is weakly log-majorized by the log-Euclidean mean and fulfills the Ando-Hiai inequality. We establish the log-majorization relationship between the R\'{e}nyi relative entropy and the product of square roots of given variables. Furthermore, we show the norm inequalities among power means and provide the boundedness of R\'{e}nyi power mean in terms of the quasi-arithmetic mean.
\end{abstract}

\medskip
\noindent \textit{2020 Mathematics Subject Classification}. 15A45, 15B48

\noindent \textit{Keywords and phrases.} Log-majorization, Cartan mean, log-Euclidean mean, Lim-P\'{a}lfia's power mean, R\'{e}nyi right mean, R\'{e}nyi power mean

\section{Introduction}

Throughout the paper, $\mathbb{C}_{m \times m}$ is the set of all $m \times m$ complex matrices,
$\mathbb{H}_m$ is the real vector space of $m \times m$ Hermitian matrices, and $\mathbb{P}_m \subset \mathbb{H}_m$ is the open convex cone of $m \times m$ positive definite matrices.
For $A, B \in \mathbb{H}_m$, the Loewner order $A \geq (>) B$ means that $A - B$ is positive semi-definite (resp. positive definite).
We denote by $s(X)$ the $m$-tuple of all singular values of a complex matrix $X$, and denote by $\lambda(X)$ the $m$-tuple of all eigenvalues of a Hermitian matrix $X$ in decreasing order: $\lambda_{1}(X) \geq \lambda_{2}(X) \geq \cdots \geq \lambda_{m}(X)$.

Let $x, y$ be two $m$-tuples of positive real numbers. We denote by $x^{\downarrow} = (x_{1}^{\downarrow}, \dots, x_{m}^{\downarrow})$ the rearrangement of $x$ in decreasing order.
The notation $x \prec_{\log} y$ represents that $x$ is \emph{log-majorized} by $y$, that is,
\begin{equation} \label{eqn:maj_eqn}
\prod_{i=1}^k x^{\downarrow}_i \leq \prod_{i=1}^k y^{\downarrow}_i
\end{equation}
for $1 \leq k \leq m-1$ and the equality holds for $k = m$.
We say that $x$ is \emph{weakly log-majorized} by $y$, denoted by $x \prec_{w \log} y$, if \eqref{eqn:maj_eqn} is true for $k = 1, 2, \dots, m$.
For simplicity, given $A, B \in \mathbb{P}_{m}$, we write $A \prec_{\log} B$ if $\lambda(A) \prec_{\log} \lambda(B)$, and $A \prec_{w \log} B$ if $\lambda(A) \prec_{w \log} \lambda(B)$.

For given $A_{1}, \dots, A_{n} \in \mathbb{P}_m$ the \emph{quasi-arithmetic mean} (\emph{generalized} or \emph{power mean}) of order $t (\neq 0) \in \mathbb{R}$ is defined by
\begin{displaymath}
\mathcal{Q}_{t}(\omega; A_1, \dots, A_n) := \left( \sum_{j=1}^{n} w_{j} A_{j}^{t} \right)^{\frac{1}{t}}
\end{displaymath}
where $\omega = (w_{1}, \dots, w_{n})$ is a positive probability vector. Note that
\begin{displaymath}
\lim_{t \to 0} \mathcal{Q}_{t}(\omega; A_1, \dots, A_n) = \exp \left( \sum_{j=1}^{n} w_{j} \log A_{j} \right),
\end{displaymath}
where the right-hand side is called the \emph{log-Euclidean mean} of $A_{1}, \dots, A_{n}$. Log-majorization properties and operator inequalities of the quasi-arithmetic mean have been studied \cite{BLY, DDF, LY13}. As non-commutative versions of the quasi-arithmetic mean, we investigate in this paper the Lim-P\'{a}lfia's power mean, R\'{e}nyi right mean and R\'{e}nyi power mean.

\noindent (I) The \emph{Lim-P\'{a}lfia's power mean} $P_{t}(\omega; A_1, \dots, A_n)$ of order $t \in (0,1]$ is defined as the unique positive definite solution of
\begin{displaymath}
X = \sum_{i=1}^{n} w_{i} (X \#_{t} A_{i}),
\end{displaymath}
where $A \#_{t} B = A^{1/2} (A^{-1/2} B A^{-1/2})^{t} A^{1/2}$ is known as the weighted geometric mean of $A, B \in \mathbb{P}_{m}$.
For $t \in [-1,0)$ we define $P_{t}(\omega; A_1, \dots, A_n) = P_{-t}(\omega; A_1^{-1}, \dots, A_n^{-1})^{-1}$. See \cite{LP12} for more information.
We show in Section 3 that the sequence $P_{t}(\omega; A_1^{p}, \dots, A_n^{p})^{1/p}$ for $t \in [-1,0)$ is weakly log-majorized by the log-Euclidean mean for any $p > 0$:
\begin{displaymath}
P_{t}(\omega; A_1^{p}, \dots, A_n^{p})^{1/p} \prec_{w \log} \exp \left( \sum_{j=1}^{n} w_{j} \log A_{j} \right),
\end{displaymath}
and the power mean $P_{t}$ satisfies the Ando-Hiai inequality: $P_{t}(\omega; A_1, \dots, A_n) \leq I$ implies $P_{t}(\omega; A_1^{p}, \dots, A_n^{p})^{1/p} \leq I$. This provides an affirmative answer for the monotone convergence of Lim-P\'{a}lfia's power means in terms of the weak log-majorization, but it is an open question:
\begin{displaymath}
P_{t}(\omega; A_1^{p}, \dots, A_n^{p})^{1/p} \nearrow_{\prec_{w \log}} \exp \left( \sum_{j=1}^{n} w_{j} \log A_{j} \right) \quad \textrm{as} \quad p \searrow 0.
\end{displaymath}

\noindent (II) Recently, a new barycenter minimizing the weighted sum of quantum divergences, called the \emph{$t$-$z$ R\'{e}nyi right mean}, has been introduced in \cite{DLVV}. Indeed, for $0 < t \leq z < 1$
\begin{displaymath}
\Omega_{t,z}(\omega; A_1, \dots, A_n) := \underset{X \in \mathbb{P}_{m}}{\arg \min} \sum_{j=1}^{n} w_{j} \Phi_{t,z}(A_{j}, X),
\end{displaymath}
where $\displaystyle \Phi_{t,z}(A, B) = \tr ((1-t) A + t B) - \tr \left( A^{\frac{1-t}{2z}} B^{\frac{t}{z}} A^{\frac{1-t}{2z}} \right)^{z}$ is the $t$-$z$ Bures-Wasserstein quantum divergence of $A, B \in \mathbb{P}_{m}$. Here, $Q_{t,z}(A, B) = \left( A^{\frac{1-t}{2z}} B^{\frac{t}{z}} A^{\frac{1-t}{2z}} \right)^{z}$ is known as the $t$-$z$ R\'{e}nyi relative entropy of $A, B$. The $t$-$z$ R\'{e}nyi right mean coincides with the unique positive definite solution of the equation
\begin{displaymath}
X = \sum_{j=1}^{n} w_{j} \left( X^{\frac{t}{2z}} A_{j}^{\frac{1-t}{z}} X^{\frac{t}{2z}} \right)^{z},
\end{displaymath}
which obtained by vanishing the gradient of objective function. For $t = z = 1/2$, the $t$-$z$  R\'{e}nyi right mean $\Omega_{t,z}$ coincides with the Wasserstein mean: see \cite{AC, ABCM, BJL19, HK19} for more information. We show in Section 4 the log-majorization relationship between the $t$-$z$ R\'{e}nyi relative entropy $Q_{t,z}(A, B)$ and $A^{1/2} B^{1/2}$, and establish norm inequalities among the power means.

\noindent (III) Dumitru and Franco \cite{DF} have defined the \emph{R\'{e}nyi power mean} $\mathcal{R}_{t,z}(\omega; A_1, \dots, A_n)$ as the unique positive definite solution of the equation
\begin{displaymath}
X = \sum_{j=1}^{n} w_{j} \left( A_{j}^{\frac{1-t}{2z}} X^{\frac{t}{z}} A_{j}^{\frac{1-t}{2z}} \right)^{z},
\end{displaymath}
and proved the norm inequality between $\mathcal{R}_{t,z}$ and $\mathcal{Q}_{1-t}$ with respect to the $p$-norm for $p \geq 2$.
Note that for commuting variables
\begin{displaymath}
\mathcal{R}_{t,z} = \Omega_{t,z} = P_{1-t} = \mathcal{Q}_{1-t}.
\end{displaymath}
We show in Section 5 the boundedness of R\'{e}nyi power mean $\mathcal{R}_{t,z}$ in terms of the quasi-arithmetic mean.

\section{Antisymmetric tensor power and homogeneous matrix means}

A crucial tool in the theory of log-majorization is the antisymmetric tensor power (or the compound matrix). Note that for $A \geq 0$ and $1 \leq k \leq m$,
\begin{equation} \label{E:anti-tensor}
\prod_{i=1}^k \lambda_i(A) = \lambda_1(\Lambda^k A),
\end{equation}
where $\Lambda^k A$ denotes the $k$th antisymmetric tensor power of $A$.
By the definition of log-majorization, $A \prec_{\log} B$ for $A, B > 0$ if and only if $\lambda_1(\Lambda^k A) \leq \lambda_1(\Lambda^k B)$ for $1 \leq k \leq m-1$, and $\det A = \det B$.
We give a list of fundamental properties of the antisymmetric tensor powers by \cite{Bh97} and \cite{Hi}.
\begin{lemma} \label{L:anti-tensor}
Let $A, B \in \mathbb{P}_{m}$, and $I$ the identity matrix with certain dimension.
\begin{itemize}
\item[(1)] $\Lambda^k (cI) = c^{k} I$ for any constant $c$
\item[(2)] $\Lambda^k(X Y) = \Lambda^k(X) \Lambda^k(Y)$ for any $X, Y \in \mathbb{C}_{m \times m}$
\item[(3)] $(\Lambda^k(A))^r = \Lambda^k(A^r)$ for any $r \in \mathbb{R}$
\item[(4)] $\Lambda^k A \leq \Lambda^k B$ whenever $A \leq B$.
\end{itemize}
\end{lemma}
Another interesting property is that the weak log-majorization implies the weak majorization. More precisely, $A \prec_{w \log} B$ implies $A \prec_{w} B$, where $A \prec_{w} B$ means that
\begin{displaymath}
\sum_{i=1}^k \lambda_i(A) \leq \sum_{i=1}^k \lambda_i(B), \quad 1 \leq k \leq m.
\end{displaymath}
Note that $A \prec_{w} B$ if and only if $||| A ||| \leq ||| B |||$ for any unitarily invariant norm $||| \cdot |||$.
One can easily see from Lemma \ref{L:anti-tensor} (4) and \eqref{E:anti-tensor} that $A \leq B$ for $A, B \in \mathbb{P}_{m}$ implies $A \prec_{w \log} B$, so $A \prec_{w} B$.

Let $\Delta_{n}$ be the simplex of all positive probability vectors in $\mathbb{R}^{n}$.
A (multi-variable) matrix mean on the open convex cone $\mathbb{P}_{m}$ is the map $G: \Delta_{n} \times \mathbb{P}_{m}^{n} \to \mathbb{P}_{m}$ satisfying the idempotency: $G(\omega; A, \dots, A) = A$ for any $\omega \in \Delta_{n}$ and $A \in \mathbb{P}_{m}$.
The matrix mean is said to be homogeneous if $G(\omega; c \mathbb{A}) = c G(\omega; \mathbb{A})$ for any $c > 0$, where $\mathbb{A} = (A_{1}, \dots, A_{n}) \in \mathbb{P}_{m}^{n}$.

\begin{lemma} \label{L:homo-mean}
Let $G_{1}, G_{2}: \Delta_{n} \times \mathbb{P}_{m}^{n} \to \mathbb{P}_{m}$ be homogeneous matrix means satisfying
\begin{equation} \label{E:suff-condition}
G_{2}(\omega; \mathbb{A}) \leq I \quad \textrm{implies} \quad G_{1}(\omega; \mathbb{A}) \leq I
\end{equation}
for any $\omega \in \Delta_{n}$ and $\mathbb{A} = (A_{1}, \dots, A_{n}) \in \mathbb{P}_{m}^{n}$. Then $\Vert G_{1}(\omega; \mathbb{A}) \Vert \leq \Vert G_{2}(\omega; \mathbb{A}) \Vert$, where $\Vert \cdot \Vert$ denotes the operator norm. In addition, if such homogeneous matrix means $G_{i}$ for $i = 1, 2$ are preserved by the antisymmetric tensor power:
\begin{displaymath}
\Lambda^{k} G_{i}(\omega; \mathbb{A}) = G_{i}(\omega; \Lambda^{k} \mathbb{A})
\end{displaymath}
where $\Lambda^{k} \mathbb{A} = (\Lambda^{k} A_{1}, \dots, \Lambda^{k} A_{n})$, then $G_{1}(\omega; \mathbb{A}) \prec_{w \log} G_{2}(\omega; \mathbb{A})$.
\end{lemma}

\begin{proof}
Let $\kappa = \Vert G_{2}(\omega; \mathbb{A}) \Vert$. Then $G_{2}(\omega; \mathbb{A}) \leq \kappa I$, and
\begin{displaymath}
G_{2} \left( \omega; \frac{1}{\kappa} \mathbb{A} \right) = \frac{1}{\kappa} G_{2}(\omega; \mathbb{A}) \leq I
\end{displaymath}
since $G_{2}$ is homogeneous. By \eqref{E:suff-condition} and the homogeneity of $G_{1}$
\begin{displaymath}
\frac{1}{\kappa} G_{1}(\omega; \mathbb{A}) = G_{1} \left( \omega; \frac{1}{\kappa} \mathbb{A} \right) \leq I.
\end{displaymath}
Thus, $G_{1}(\omega; \mathbb{A}) \leq \kappa I$, which implies $\Vert G_{1}(\omega; \mathbb{A}) \Vert \leq \Vert G_{2}(\omega; \mathbb{A}) \Vert$.

Additionally, assume that $G_{i}$ for $i = 1, 2$ are preserved by the antisymmetric tensor power. Then using fundamental properties of the antisymmetric tensor powers in Lemma \ref{L:anti-tensor}, \eqref{E:suff-condition} yields
\begin{displaymath}
\Lambda^{k} G_{2}(\omega; \mathbb{A}) \leq I \quad \Longrightarrow \quad \Lambda^{k} G_{1}(\omega; \mathbb{A}) \leq I.
\end{displaymath}
So $\lambda_{1} (\Lambda^{k} G_{1}(\omega; \mathbb{A})) \leq \lambda_{1} (\Lambda^{k} G_{2}(\omega; \mathbb{A}))$, equivalently $G_{1}(\omega; \mathbb{A}) \prec_{w \log} G_{2}(\omega; \mathbb{A})$.
\end{proof}

\section{Log-majorization of the Lim-P\'{a}lfia's power mean}

Let $\mathbb{A} = (A_{1}, \dots, A_{n}) \in \mathbb{P}_{m}^{n}$.
For convenience, we denote
\begin{displaymath}
\begin{split}
\mathbb{A}^{p} & := (A_{1}^{p}, \dots, A_{n}^{p}) \in \mathbb{P}_{m}^{n}
\end{split}
\end{displaymath}
for any $p \in \mathbb{R}$.

For $t \in (0,1]$ we denote by $P_{t}(\omega; \mathbb{A})$ the unique positive definite solution of
\begin{displaymath}
X = \sum_{i=1}^{n} w_{i} (X \#_{t} A_{i}).
\end{displaymath}
For $t \in [-1,0)$ we define $P_{t}(\omega; \mathbb{A}) = P_{-t}(\omega; \mathbb{A}^{-1})^{-1}$. We call $P_{t}(\omega; \mathbb{A})$ the \emph{Lim-P\'{a}lfia's power mean} of order $t$ for $A_{1}, \dots, A_{n}$.
Note that
\begin{center}
$\displaystyle P_{1}(\omega; \mathbb{A}) = \sum_{j=1}^{n} w_{j} A_{j} = \mathcal{A}(\omega; \mathbb{A}) \quad$ and $\quad \displaystyle P_{-1}(\omega; \mathbb{A}) = \left( \sum_{j=1}^{n} w_{j} A_{j}^{-1} \right)^{-1} = \mathcal{H}(\omega; \mathbb{A})$,
\end{center}
where $\mathcal{A}$ and $\mathcal{H}$ denote the arithmetic and harmonic means respectively.
One can easily see that for commuting $A_{1}, \dots, A_{n}$
\begin{displaymath}
P_{t}(\omega; \mathbb{A}) = \left( \sum_{i=1}^{n} w_{i} A_{i}^{t} \right)^{1/t} = \mathcal{Q}_{t}(\omega; \mathbb{A}),
\end{displaymath}
where $\mathcal{Q}_{t}$ denotes the quasi-arithmetic mean of order $t$; it can be defined for all $t \in \mathbb{R}$, and
\begin{displaymath}
\lim_{t \to 0} \mathcal{Q}_{t}(\omega; \mathbb{A}) = \exp \left( \sum_{i=1}^{n} w_{i} \log A_{i} \right).
\end{displaymath}
The remarkable consequence of power means appeared in \cite{LL14, LP12} is that $P_{t}$ converges monotonically to the Cartan mean $\Lambda$ as $t \to 0$ such that
\begin{equation} \label{E:para-mono}
P_{-t} \leq P_{-s} \leq \cdots \leq \Lambda = \lim_{t \to 0} P_{t} \leq \cdots \leq P_{s} \leq P_{t}
\end{equation}
for $0 < s \leq t \leq 1$, where the Cartan mean $\Lambda$ is the least squares mean for the Riemannian trace metric $d_{R}$:
\begin{displaymath}
\Lambda(\omega; A_1, \dots, A_n) := \underset{X \in \mathbb{P}_{m}}{\arg \min} \sum_{j=1}^{n} w_{j} d_{R}^{2}(A_{j}, X),
\end{displaymath}
and $d_{R}(A, B) = \Vert \log A^{-1/2} B A^{-1/2} \Vert_{2}$.

\begin{remark}
Note that Lim-P\'{a}lfia's power mean and Cartan mean are homogeneous.
So applying Lemma \ref{L:homo-mean} with the monotonicity \eqref{E:para-mono} of Lim-P\'{a}lfia's power means yields that
\begin{displaymath}
\begin{split}
P_{t}(\omega; \mathbb{A}) \ & \searrow_{\prec_{w \log}} \ \Lambda(\omega; \mathbb{A}) \qquad \textrm{as} \qquad t \searrow 0, \\
P_{t}(\omega; \mathbb{A}) \ & \nearrow_{\prec_{w \log}} \ \Lambda(\omega; \mathbb{A}) \qquad \textrm{as} \qquad t \nearrow 0.
\end{split}
\end{displaymath}
\end{remark}

\begin{theorem} \label{T:Yamazaki} \cite[Theorem 1]{Ya12}
Let $\mathbb{A} = (A_{1}, \dots, A_{n}) \in \mathbb{P}_{m}^{n}$ and $\omega = (w_{1}, \dots, w_{n}) \in \Delta_{n}$. Then
\begin{center}
$\displaystyle \sum_{j=1}^{n} w_{j} \log A_{j} \leq 0 \qquad$ implies $\qquad \Lambda(\omega; \mathbb{A}) \leq I$.
\end{center}
\end{theorem}

\begin{proposition} \label{P:power-means}
Let $\mathbb{A} = (A_{1}, \dots, A_{n}) \in \mathbb{P}_{m}^{n}$, $\omega = (w_{1}, \dots, w_{n}) \in \Delta_{n}$, and $0 < t \leq 1$. Then for any $p > 0$
\begin{equation} \label{E:op-norm}
\Vert P_{-t}(\omega; \mathbb{A}^{p})^{1/p} \Vert \leq \left\Vert \exp \left( \sum_{j=1}^{n} w_{j} \log A_{j} \right) \right\Vert \leq \Vert P_{t}(\omega; \mathbb{A}^{p})^{1/p} \Vert.
\end{equation}
Furthermore,
\begin{equation} \label{E:pmean-major}
P_{-t}(\omega; \mathbb{A}^{p})^{1/p} \prec_{w \log} \exp \left( \sum_{j=1}^{n} w_{j} \log A_{j} \right).
\end{equation}
\end{proposition}

\begin{proof}
Let $p > 0$. Since the Lim-P\'{a}lfia's power mean and log-Euclidean mean are homogeneous, by Lemma \ref{L:homo-mean} it is enough for the second inequality of \eqref{E:op-norm} to show that for $0 < t \leq 1$
\begin{center}
$P_{t}(\omega; \mathbb{A}^{p})^{1/p} \leq I$ \qquad implies \qquad $\displaystyle \exp \left( \sum_{j=1}^{n} w_{j} \log A_{j} \right) \leq I$.
\end{center}
Assume that $P_{t}(\omega; \mathbb{A}^{p}) \leq I$ for $0 < t \leq 1$. By \eqref{E:para-mono} $\Lambda(\omega; \mathbb{A}^{p}) \leq I$, and $\Lambda(\omega; \mathbb{A}^{p})^{1/p} \leq I$. Taking the limit as $p \to 0^+$ implies that
\begin{displaymath}
\exp \left( \sum_{j=1}^{n} w_{j} \log A_{j} \right) \leq I.
\end{displaymath}
Now assume that $\displaystyle \exp \left( \sum_{j=1}^{n} w_{j} \log A_{j} \right) \leq I$. Since the logarithmic map is operator monotone, we have $\displaystyle \sum_{j=1}^{n} w_{j} \log A_{j} \leq 0$. Then $\displaystyle \sum_{j=1}^{n} w_{j} \log A_{j}^{p} = p \sum_{j=1}^{n} w_{j} \log A_{j} \leq 0$ for any $p > 0$. By Theorem \ref{T:Yamazaki}
\begin{displaymath}
\Lambda(\omega; \mathbb{A}^{p}) \leq I,
\end{displaymath}
and by \eqref{E:para-mono} $P_{-t}(\omega; \mathbb{A}^{p}) \leq I$ for $0 < t \leq 1$. This completes the proof of \eqref{E:op-norm}.

Furthermore, by \eqref{E:para-mono} $\Lambda^{k} P_{-t}(\omega; \mathbb{A}^{p}) \leq \Lambda^{k} \Lambda(\omega; \mathbb{A}^{p})$ for the $k$th antisymmetric tensor power $\Lambda^{k}$. So $\lambda_{1} (\Lambda^{k} P_{-t}(\omega; \mathbb{A}^{p})) \leq \lambda_{1} (\Lambda^{k} \Lambda(\omega; \mathbb{A}^{p}))$, and by Lemma \ref{L:anti-tensor} (3)
\begin{displaymath}
\lambda_{1} (\Lambda^{k} P_{-t}(\omega; \mathbb{A}^{p})^{1/p}) = \lambda_{1} (\Lambda^{k} P_{-t}(\omega; \mathbb{A}^{p}))^{1/p} \leq \lambda_{1} (\Lambda^{k} \Lambda(\omega; \mathbb{A}^{p}))^{1/p} = \lambda_{1} (\Lambda^{k} \Lambda(\omega; \mathbb{A}^{p})^{1/p}).
\end{displaymath}
Since $\displaystyle \Lambda(\omega; \mathbb{A}) \prec_{\log} \exp \left( \sum_{j=1}^{n} w_{j} \log A_{j} \right)$ by \cite[Theorem 1]{BJL19}, we conclude that
\begin{displaymath}
P_{-t}(\omega; \mathbb{A}^{p})^{1/p} \prec_{w \log} \Lambda(\omega; \mathbb{A}^{p})^{1/p} \prec_{\log} \exp \left( \sum_{j=1}^{n} w_{j} \log A_{j} \right).
\end{displaymath}
\end{proof}

\begin{remark}
Note from \cite[Proposition 3.5]{LP12} that for $t \in (0,1]$
\begin{displaymath}
\det P_{-t}(\omega; \mathbb{A}) \leq \prod_{j=1}^{n} (\det A_{j})^{w_{j}},
\end{displaymath}
so \eqref{E:pmean-major} must be the weak log-majorization.
\end{remark}

A variant of Ando-Hiai inequality for power means has been shown in \cite[Corollary 3.2]{LY13}: for $t \in (0,1]$
\begin{center}
$P_{t}(\omega; \mathbb{A}) \leq I \quad$ implies $\quad P_{\frac{t}{p}}(\omega; \mathbb{A}^{p}) \leq I \quad$ for all $p \geq 1$.
\end{center}
We provide different types of Ando-Hiai inequality for power means using Jensen inequalities \cite{HP}.
Let $X$ be a contraction. For any $A > 0$ we have
\begin{equation} \label{E:Hansen1}
(X A X^{*})^{p} \leq X A^{p} X^{*} \hspace{5mm} \textrm{if} \ \ 1 \leq p \leq 2,
\end{equation}
and
\begin{equation} \label{E:Hansen2}
(X A X^{*})^{p} \geq X A^{p} X^{*} \hspace{5mm} \textrm{if} \ \ 0 \leq p \leq 1.
\end{equation}

\begin{theorem} \label{T:power-means}
Let $p \geq 1$. Then
\begin{itemize}
\item[(i)] if $P_{t}(\omega; \mathbb{A}) \geq I$ then $P_{t}(\omega; \mathbb{A}) \leq P_{t}(\omega; \mathbb{A}^{p})$ for $0 < t \leq 1$, and
\item[(ii)] if $P_{t}(\omega; \mathbb{A}) \leq I$ then $P_{t}(\omega; \mathbb{A}) \geq P_{t}(\omega; \mathbb{A}^{p})$ for $-1 \leq t < 0$.
\end{itemize}
\end{theorem}

\begin{proof}
We first consider $1 \leq p \leq 2$. Assume that $X := P_{t}(\omega; \mathbb{A}) \geq I$ for $0 < t \leq 1$. Then by taking the congruence transformation
\begin{displaymath}
I = \sum_{j=1}^{n} w_{j} (X^{-1/2} A_{j} X^{-1/2})^{t} = \sum_{j=1}^{n} w_{j} \left[ (X^{-1/2} A_{j} X^{-1/2})^{p} \right]^{t/p}.
\end{displaymath}
Since $0 < t/p \leq 1$, the above identity reduces to
\begin{displaymath}
I = P_{t/p}(\omega; (X^{-1/2} A_{1} X^{-1/2})^{p}, \dots, (X^{-1/2} A_{n} X^{-1/2})^{p}).
\end{displaymath}
Since $X^{-1/2} \leq I$, Hansen's inequality \eqref{E:Hansen1} and the monotonicity of power means yield
\begin{displaymath}
I \leq P_{t/p}(\omega; X^{-1/2} A_{1}^{p} X^{-1/2}, \dots, X^{-1/2} A_{n}^{p} X^{-1/2}).
\end{displaymath}
Taking the congruence transformation by $X^{1/2}$ implies that $X \leq P_{t/p}(\omega; \mathbb{A}^{p})$. Since $0 < t/p \leq t \leq 1$ we obtain from \eqref{E:para-mono}
\begin{displaymath}
X \leq P_{t/p}(\omega; \mathbb{A}^{p}) \leq P_{t}(\omega; \mathbb{A}^{p}).
\end{displaymath}
Replacing $A_{j}$ by $A_{j}^{2}$ we can extend the interval $[2,4]$, and successfully for all $p \geq 1$.

Assume that $X := P_{t}(\omega; \mathbb{A}) \leq I$ for $-1 \leq t < 0$. Then $X^{-1} = P_{-t}(\omega; \mathbb{A}^{-1}) \geq I$. By (i) with $0 < -t \leq 1$
\begin{displaymath}
X^{-1} \leq P_{-t}(\omega; \mathbb{A}^{-p}),
\end{displaymath}
equivalently, $X \geq P_{-t}(\omega; \mathbb{A}^{-p})^{-1} = P_{t}(\omega; \mathbb{A}^{p})$.
\end{proof}

\begin{remark}
We give another proof for Theorem \ref{T:power-means} (i). Let $1 \leq p \leq 2$. Assume that $X = P_{t}(\omega; \mathbb{A}) \geq I$ for $0 < t \leq 1$. Since the map $A \in \mathbb{P}_{m} \mapsto A^{p}$ is operator convex,
\begin{displaymath}
I = \left[ \sum_{j=1}^{n} w_{j} (X^{-1/2} A_{j} X^{-1/2})^{t} \right]^{p} \leq \sum_{j=1}^{n} w_{j} (X^{-1/2} A_{j} X^{-1/2})^{pt}.
\end{displaymath}
By \eqref{E:Hansen1} and the monotonicity of the power map $A \in \mathbb{P}_{m} \mapsto A^{t}$,
\begin{displaymath}
I \leq \sum_{j=1}^{n} w_{j} (X^{-1/2} A_{j}^{p} X^{-1/2})^{t}.
\end{displaymath}
Taking congruence transformation by $X^{1/2}$ implies
\begin{displaymath}
X \leq \sum_{j=1}^{n} w_{j} X^{1/2} (X^{-1/2} A_{j}^{p} X^{-1/2})^{t} X^{1/2} = \sum_{j=1}^{n} w_{j} X \#_{t} A_{j}^{p} =: f(X).
\end{displaymath}
Since the map $f$ is operator monotone on $\mathbb{P}_{m}$, we have $X \leq f(X) \leq f^{2}(X) \leq \cdots \leq f^{k}(X)$ for all $k \geq 1$. Taking the limit as $k \to \infty$ yields $X \leq P_{t}(\omega; \mathbb{A}^{p})$ for $1 \leq p \leq 2$. Replacing $A_{j}$ by $A_{j}^{2}$ we can extend the interval $[2,4]$, and successfully for all $p \geq 1$.
\end{remark}

Applying Lemma \ref{L:homo-mean} to Theorem \ref{T:power-means} (ii) we obtain
\begin{corollary}
Let $-1 \leq t < 0$. Then
\begin{displaymath}
\Vert P_{t}(\omega; \mathbb{A}^{p})^{1/p} \Vert \leq \Vert P_{t}(\omega; \mathbb{A}) \Vert
\end{displaymath}
for $p \geq 1$, where $\Vert \cdot \Vert$ denotes the operator norm.
\end{corollary}


\begin{remark}
The following is the unique characterization of the Cartan mean among other multi-variable geometric means satisfying the Ando-Li-Mathias axioms:
\begin{equation} \label{E:Cartan}
\Lambda(\omega; \mathbb{A}) \leq I \qquad \textrm{implies} \qquad \Lambda(\omega; \mathbb{A}^{p}) \leq I
\end{equation}
for all $p \geq 1$. This is known as the Ando-Hiai inequality; see \cite[Theorem 3, Corollary 6]{Ya12}. We can derive it by using Theorem \ref{T:power-means} (ii). Indeed, assume that $\Lambda(\omega; \mathbb{A}) \leq I$. Then by \eqref{E:para-mono} $P_{t}(\omega; \mathbb{A}) \leq I$ for $-1 \leq t < 0$, and by Theorem \ref{T:power-means} (ii) $P_{t}(\omega; \mathbb{A}^{p}) \leq I$. Taking the limit as $t \to 0^-$ yields $\Lambda(\omega; \mathbb{A}^{p}) \leq I$.
\end{remark}

\begin{theorem} \label{T:Cartan}
Let $\mathbb{A} = (A_{1}, \dots, A_{n}) \in \mathbb{P}_{m}^{n}$, and $\omega = (w_{1}, \dots, w_{n}) \in \Delta_{n}$. Then
\begin{displaymath}
\Lambda(\omega; \mathbb{A}^{p})^{1/p} \nearrow_{\prec_{\log}} \exp \left( \sum_{j=1}^{n} w_{j} \log A_{j} \right) \quad \textrm{as} \quad p \searrow 0.
\end{displaymath}
\end{theorem}

\begin{proof}
Note from \cite{BJL19} that
\begin{displaymath}
\lim_{p \to 0} \Lambda(\omega; \mathbb{A}^{p})^{1/p} = \exp \left( \sum_{j=1}^{n} w_{j} \log A_{j} \right),
\end{displaymath}
and
\begin{displaymath}
\Lambda(\omega; \mathbb{A}^{p})^{1/p} \prec_{\log} \exp \left( \sum_{j=1}^{n} w_{j} \log A_{j} \right).
\end{displaymath}
So it is enough to show that $\Lambda(\omega; \mathbb{A}^{q})^{1/q} \prec_{\log} \Lambda(\omega; \mathbb{A}^{p})^{1/p}$ for $0 < p \leq q$.
By \eqref{E:Cartan}, if $\Lambda(\omega; \mathbb{A}) \leq I$ then $\Lambda(\omega; \mathbb{A}^{r}) \leq I$ for any $r \geq 1$ so $\Lambda(\omega; \mathbb{A}^{r})^{1/r} \leq I$.

Since the Cartan mean and $\Lambda(\omega; \mathbb{A}^{r})^{1/r}$ are preserved by the antisymmetric tensor power and homogeneous, from Lemma \ref{L:homo-mean} we have
$\Lambda(\omega; \mathbb{A}^{r})^{1/r} \prec_{\log} \Lambda(\omega; \mathbb{A})$ for all $r \geq 1$.
Letting $r = q/p$ for $0 < p \leq q$ and replacing $A_{j}$ by $A_{j}^{p}$ we obtain $\Lambda(\omega; \mathbb{A}^{q})^{1/q} \prec_{\log} \Lambda(\omega; \mathbb{A}^{p})^{1/p}$.
\end{proof}

\begin{remark}
Ando and Hiai \cite{An94, AH} have shown that $(A^{p} \#_{t} B^{p})^{1/p}$ converges increasingly to the log-Euclidean mean as $p \to 0^{+}$ with respect to the log-majorization:
\begin{displaymath}
(A^{p} \#_{t} B^{p})^{1/p} \nearrow_{\prec_{\log}} \exp ((1-t) \log A + t \log B) \quad \textrm{as} \quad p \searrow 0.
\end{displaymath}
Theorem \ref{T:Cartan} is a generalization of the Ando-Hiai's log-majorization result to multi-variable geometric mean, which is the Cartan mean.
\end{remark}

\begin{remark}
Since the Lim-P\'{a}lfia's power mean satisfies the arithmetic-power-harmonic mean inequalities:
\begin{displaymath}
\mathcal{H}(\omega; \mathbb{A}) = \left( \sum_{j=1}^{n} w_{j} A_{j}^{-1} \right)^{-1} \leq P_{t}(\omega; \mathbb{A}) \leq \sum_{j=1}^{n} w_{j} A_{j} = \mathcal{A}(\omega; \mathbb{A})
\end{displaymath}
for any nonzero $t \in [-1,1]$, it satisfies from \cite[Theorem 4.2]{HK17}
\begin{displaymath}
\lim_{p \to 0} P_{t}(\omega; \mathbb{A}^{p})^{1/p} = \exp \left( \sum_{j=1}^{n} w_{j} \log A_{j} \right).
\end{displaymath}
Moreover, $\displaystyle P_{t}(\omega; \mathbb{A}^{p})^{1/p} \prec_{w \log} \exp \left( \sum_{j=1}^{n} w_{j} \log A_{j} \right)$ for $t \in [-1,0)$ by Proposition \ref{P:power-means}. One can naturally ask that the Lim-P\'{a}lfia's power mean $P_{t}(\omega; \mathbb{A}^{p})^{1/p}$ for $t \in [-1,0)$ converges increasingly to the log-Euclidean mean as $p \to 0^{+}$ with respect to the weak log-majorization. In order to show this, it remains an open problem as follows: for $0 < p \leq q$
\begin{displaymath}
P_{t}(\omega; \mathbb{A}^{q})^{1/q} \prec_{w \log} P_{t}(\omega; \mathbb{A}^{p})^{1/p}.
\end{displaymath}
\end{remark}

\section{Log-majorization of the $t$-$z$ R\'{e}nyi right mean}

Let $A, B \in \mathbb{P}_{m}$. For $0 \leq t \leq 1$ and $z > 0$
\begin{displaymath}
Q_{t,z}(A, B) = \left( A^{\frac{1-t}{2z}} B^{\frac{t}{z}} A^{\frac{1-t}{2z}} \right)^{z}
\end{displaymath}
is the matrix version of the $t$-$z$ R\'{e}nyi relative entropy \cite{AD, MO}. Especially, $Q_{t,t}(A, B)$ is known as the sandwiched R\'{e}nyi relative entropy \cite{WWY}. This can be considered as a non-commutative version of geometric mean in the sense that $Q_{t,z}(A, B) = A^{1-t} B^{t}$ for commuting $A$ and $B$. From this point of view it is interesting to find a log-majorization relation between $Q_{t,z}(A, B)$ and $A^{1/2} B^{1/2}$.

\begin{theorem}
Let $A, B \in \mathbb{P}_{m}$. For $0 \leq t \leq 1/2$ and $t \leq z \leq 1$,
\begin{itemize}
  \item[(i)] $\lambda (Q_{t,z}(A, B)) \prec_{\log} s(A^{1/2} B^{1/2})$, and
  \item[(ii)] $s(A^{t - \frac{1}{2}} Q_{t,z}(A, B) B^{\frac{1}{2} - t}) \prec_{\log} s(A^{1/2} B^{1/2})$.
\end{itemize}
\end{theorem}

\begin{proof}
Note that $s(A^{1/2} B^{1/2}) = \lambda((A^{1/2} B A^{1/2})^{1/2})$. Since $Q_{t,z}(A, B)$ and $(A^{1/2} B A^{1/2})^{1/2}$ are invariant under the antisymmetric tensor product and homogeneous, it is enough from Lemma \ref{L:homo-mean} to show that
\begin{itemize}
  \item[(i)] $A^{1/2} B A^{1/2} \leq I \qquad$ implies $\qquad Q_{t,z}(A, B) \leq I$,
  \item[(ii)] $A^{1/2} B A^{1/2} \leq I \qquad$ implies $\qquad A^{t - \frac{1}{2}} Q_{t,z}(A, B) B^{1-2t} Q_{t,z}(A, B) A^{t - \frac{1}{2}} \leq I$.
\end{itemize}

Let $0 \leq t \leq 1/2$ and $t \leq z \leq 1$.

\noindent (i) We first prove it when $B \geq I$. Assuming that $A^{1/2} B A^{1/2} \leq I$, we have $B \leq A^{-1}$ so $B^{\frac{t}{z}} \leq A^{-\frac{t}{z}}$ by the Loewner-Heinz inequality with $0 < t \leq z \leq 1$. Then
\begin{displaymath}
Q_{t,z}(A, B) \leq \left( A^{\frac{1-t}{2z}} A^{\frac{t}{z}} A^{\frac{1-t}{2z}} \right)^{z} = A^{1-2t} \leq I
\end{displaymath}
since $A \leq B^{-1} \leq I$ and $1-2t \geq 0$. So (i) holds when $B \geq I$.

Let $\lambda_{m} := \min \{ \lambda_{i}(B): 1 \leq i \leq m \}$. Then $\lambda_{m}^{-1} B \geq I$. By the preceding argument
\begin{displaymath}
\begin{split}
& \lambda_{m}^{-1} Q_{t,z}(A, B) = Q_{t,z}(\lambda_{m}^{-1} A, \lambda_{m}^{-1} B) \\
& \prec_{\log} ((\lambda_{m}^{-1} A)^{1/2} (\lambda_{m}^{-1} B) (\lambda_{m}^{-1} A)^{1/2})^{1/2} = \lambda_{m}^{-1} (A^{1/2} B A^{1/2})^{1/2},
\end{split}
\end{displaymath}
which completes the proof of (i).

\noindent (ii) Assume that $A^{1/2} B A^{1/2} \leq I$. Then $B \leq A^{-1}$, and $B^{1-2t} \leq A^{2t-1}$ by the Loewner-Heinz inequality since $2t \in [0,1]$. Therefore we have
\begin{displaymath}
\begin{split}
A^{t - \frac{1}{2}} Q_{t,z}(A, B) B^{1-2t} Q_{t,z}(A, B) A^{t - \frac{1}{2}} & \leq A^{t - \frac{1}{2}} Q_{t,z}(A, B) A^{2t-1} Q_{t,z}(A, B) A^{t - \frac{1}{2}} \\
& = \left( A^{t - \frac{1}{2}} Q_{t,z}(A, B) A^{t - \frac{1}{2}} \right)^{2}.
\end{split}
\end{displaymath}
Since $B \leq A^{-1}$, we obtain $Q_{t,z}(A, B) \leq A^{1-2t}$ by the Loewner-Heinz inequality since $0 \leq t \leq z \leq 1$. So
\begin{displaymath}
A^{t - \frac{1}{2}} Q_{t,z}(A, B) A^{t - \frac{1}{2}} \leq I,
\end{displaymath}
and thus, $A^{t - \frac{1}{2}} Q_{t,z}(A, B) B^{1-2t} Q_{t,z}(A, B) A^{t - \frac{1}{2}} \leq I$. Moreover,
\begin{displaymath}
\det \left[ A^{t - \frac{1}{2}} Q_{t,z}(A, B) B^{1-2t} Q_{t,z}(A, B) A^{t - \frac{1}{2}} \right] = \det (A B) = \det (A^{1/2} B A^{1/2}),
\end{displaymath}
and hence, (ii) holds for $t \in [0, 1/2]$.
\end{proof}

The $t$-$z$ R\'{e}nyi right mean $\Omega_{t,z}$ is defined as
\begin{displaymath}
\Omega_{t,z}(\omega; \mathbb{A}) = \underset{X \in \mathbb{P}_{m}}{\arg \min} \sum_{j=1}^{n} w_{j} \Phi_{t,z}(A_{j}, X).
\end{displaymath}
Since the map $A \in \mathbb{P}_{m} \mapsto \tr A^{t}$ for $t \in (0,1)$ is strictly concave, the map $X \in \mathbb{P}_{m} \mapsto \tr \Phi_{t,z}(A,X)$ is strictly concave for $0 < t \leq z < 1$. So one can see that $\Omega_{t,z}(\omega; \mathbb{A})$ coincides with the unique positive definite solution of the matrix nonlinear equation
\begin{equation} \label{E:Renyi}
X = \sum_{j=1}^{n} w_{j} Q_{1-t, z}(X, A_{j}).
\end{equation}
Note that \eqref{E:Renyi} is equivalent to
\begin{equation} \label{E:matrix equation}
X^{1 - \frac{t}{z}} = \sum_{j=1}^{n} w_{j} X^{- \frac{t}{z}} \#_{z} A_{j}^{\frac{1-t}{z}}.
\end{equation}
See \cite{DLVV,HJK21,JHK} for more details.

\begin{theorem} \cite[Theorem 3.2]{JHK} \label{T:A-R ineq 2}
Let $0 < t \leq z < 1$. If $\Omega_{t,z}(\omega; \mathbb{A}) \leq I$ then
\begin{displaymath}
\Omega_{t,z}(\omega; \mathbb{A})^{1 - \frac{t}{z}} \geq \mathcal{A}(\omega; \mathbb{A}^{1-t}).
\end{displaymath}
If $\Omega_{t,z}(\omega; \mathbb{A}) \geq I$ then the reverse inequality holds.
\end{theorem}

\begin{theorem} \cite[Theorem 13]{DLV} \label{T:Lower&UpperBound of R}
Let $0 < t \leq z < 1$. Then we have
\begin{displaymath}
\frac{1+z-t}{1-t}I - \frac{z}{1-t} \sum_{j=1}^{n} w_j A_j^{-\frac{1-t}{z}} \leq \Omega_{t,z}(\omega; \mathbb{A}) \leq \left( \frac{1+z-t}{1-t}I - \frac{z}{1-t} \sum_{j=1}^{n} w_j A_j^{\frac{1-t}{z}} \right)^{-1},
\end{displaymath}
where the second inequality holds when $\displaystyle(1+z-t)I - z\sum_{j=1}^{n} w_j A_j^{\frac{1-t}{z}}$ is invertible.
\end{theorem}

\begin{theorem} \label{T:Renyi-power}
For $0 < t \leq z < 1$,
\begin{displaymath}
\Vert P_{1-t}(\omega; \mathbb{A}) \Vert \leq \Vert \Omega_{t,z}(\omega; \mathbb{A}) \Vert \leq \Vert \mathcal{Q}_{\frac{1-t}{z}}(\omega;\mathbb{A}) \Vert.
\end{displaymath}
Furthermore, $\displaystyle \Vert \mathcal{Q}_{\frac{t-1}{z}}(\omega;\mathbb{A}) \Vert \leq \Vert \Omega_{t,z}(\omega; \mathbb{A}) \Vert$.
\end{theorem}

\begin{proof}
Let $0 < t \leq z < 1$. 
Since the R\'{e}nyi right mean $\Omega_{t,z}$, power mean $P_{1-t}$ and quasi-arithmetic mean $\mathcal{Q}_{\frac{1-t}{z}}$ are all homogeneous, it is enough from Lemma \ref{L:homo-mean} to show that for each cases
\begin{center}
$\Omega_{t,z}(\omega; \mathbb{A}) \leq I \quad$ implies $\quad P_{1-t}(\omega; \mathbb{A}) \leq I$, \\
$\mathcal{Q}_{\frac{1-t}{z}}(\omega;\mathbb{A}) \leq I \quad$ implies $\quad \Omega_{t,z}(\omega; \mathbb{A}) \leq I$.
\end{center}

By Theorem \ref{T:A-R ineq 2}, $\Omega_{t,z}(\omega; \mathbb{A}) \leq I$ implies that
\begin{displaymath}
\sum_{j=1}^{n} w_{j} A_{j}^{1-t} \leq \Omega_{t,z}(\omega; \mathbb{A})^{1 - \frac{t}{z}} \leq I,
\end{displaymath}
and hence, $\displaystyle \mathcal{Q}_{1-t}(\omega; \mathbb{A}) = \left( \sum_{j=1}^{n} w_{j} A_{j}^{1-t} \right)^{\frac{1}{1-t}} \leq I$.
By \cite[Theorem 3.1]{LY13} $P_{1-t}(\omega; \mathbb{A}) \leq I$.
Next, we assume $\mathcal{Q}_{\frac{1-t}{z}}(\omega;\mathbb{A}) \leq I$.
Then $\displaystyle\sum_{j=1}^{n} w_j A_j^{\frac{1-t}{z}} \leq I$ so one can see that
\begin{align*}
\frac{1+z-t}{1-t}I - \frac{z}{1-t} \sum_{j=1}^{n} w_j A_j^{\frac{1-t}{z}} \geq I.
\end{align*}
By assumption, the second inequality in Theorem \ref{T:Lower&UpperBound of R} holds, and hence, we have
\begin{displaymath}
\Omega_{t,z}(\omega; \mathbb{A})^{\frac{1-t}{z}} \leq \left( \frac{1+z-t}{1-t}I - \frac{z}{1-t} \sum_{j=1}^{n} w_j A_j^{\frac{1-t}{z}} \right)^{-1} \leq I.
\end{displaymath}

Moreover, assuming that $\Omega_{t,z}(\omega; \mathbb{A}) \leq I$ yields
\begin{displaymath}
\frac{1+z-t}{1-t}I - \frac{z}{1-t} \sum_{j=1}^{n} w_j A_j^{-\frac{1-t}{z}} \leq I
\end{displaymath}
by Theorem \ref{T:Lower&UpperBound of R}. Then $\displaystyle \sum_{j=1}^{n} w_j A_j^{\frac{t-1}{z}} \geq I$, and hence, $\displaystyle \mathcal{Q}_{\frac{t-1}{z}}(\omega;\mathbb{A}) \leq I$ since $t \in (0,1)$. This completes the proof.
\end{proof}

\section{Boundedness of the R\'{e}nyi power mean}

Another type of the R\'{e}nyi power mean has been introduced in \cite{DF}, as a unique positive definite solution of the equation
\begin{equation} \label{E:Renyi-power}
X = \sum_{j=1}^{n} w_{j} Q_{t, z}(A_{j}, X) = \sum_{j=1}^{n} w_{j} \left( A_{j}^{\frac{1-t}{2z}} X^{\frac{t}{z}} A_{j}^{\frac{1-t}{2z}} \right)^{z}.
\end{equation}
We denote it as $\mathcal{R}_{t,z}(\omega; \mathbb{A})$. We see the inequalities between the R\'{e}nyi power mean and quasi-arithmetic mean by using Jensen type inequalities.

\begin{theorem} \label{T:R-Q}
Let $0 < t \leq z < 1$. If $\mathcal{R}_{t,z}(\omega; \mathbb{A}) \leq I$ then
\begin{displaymath}
\mathcal{R}_{t,z}(\omega; \mathbb{A}) \leq \mathcal{Q}_{\frac{1}{p}}(\omega; \mathbb{A}^{1-t}) = \left( \sum_{j=1}^{n} w_{j} A_{j}^{\frac{1-t}{p}} \right)^{p}
\end{displaymath}
for all $p$ such that $p \leq z$.
\end{theorem}

\begin{proof}
Let $X = \mathcal{R}_{t,z}(\omega; \mathbb{A}) \leq I$ for $0 < t \leq z < 1$.
Since $X^{\frac{t}{z}} \leq I$, we have $A_j^{\frac{1-t}{2z}} X^{\frac{t}{z}} A_j^{\frac{1-t}{2z}} \leq A_j^{\frac{1-t}{z}}$ for each $j=1, \dots, n$.
Then from the equation \eqref{E:Renyi-power}
\begin{displaymath}
X = \sum_{j=1}^{n} w_j \left(A_j^{\frac{1-t}{2z}} X^{\frac{t}{z}} A_j^{\frac{1-t}{2z}}\right)^z
\leq \sum_{j=1}^{n} w_j A_j^{1-t}.
\end{displaymath}
Since the map $\mathbb{P}_m \ni A \mapsto A^z$ is concave, we obtain
\begin{displaymath}
X \leq \sum_{j=1}^{n} w_j A_j^{1-t} \leq \left[ \sum_{j=1}^{n} w_j A_j^{\frac{1-t}{z}} \right]^z = \mathcal{Q}_{\frac{1}{z}}(\omega;\mathbb{A}^{1-t}).
\end{displaymath}
Moreover, $\mathcal{Q}_p$ is monotone on $p \in (-\infty, -1] \cup [1, \infty)$ from \cite[Theorem 5.1]{Kim18} so
\begin{displaymath}
\mathcal{Q}_{\frac{1}{z}}(\omega;\mathbb{A}^{1-t}) \leq \mathcal{Q}_{\frac{1}{p}}(\omega; \mathbb{A}^{1-t}),
\end{displaymath}
for $0 < p \leq z < 1$. Hence, we completes the proof.
\end{proof}

\begin{lemma} \label{L:supplement}
Let $0 < t \leq z < 1$.
\begin{itemize}
\item[(1)] If $A_{j} \leq I$ for all $j$, then $\mathcal{R}_{t,z}(\omega; \mathbb{A}) \leq I$.
\item[(2)] If $A_{j} \geq I$ for all $j$, then $\mathcal{R}_{t,z}(\omega; \mathbb{A}) \geq I$.
\end{itemize}
\end{lemma}

\begin{proof}
Assume that $A_{j} \leq I$ for all $j$. Let $X = \mathcal{R}_{t,z}(\omega; \mathbb{A})$ for $0 < t \leq z < 1$. Suppose that $\lambda_{1}(X) > 1$. Since $X \leq \lambda_{1}(X) I$,
\begin{displaymath}
X = \sum_{j=1}^{n} w_{j} \left( A_{j}^{\frac{1-t}{2z}} X^{\frac{t}{z}} A_{j}^{\frac{1-t}{2z}} \right)^{z} \leq \lambda_{1}(X)^{t} \sum_{j=1}^{n} w_{j} A_{j}^{1-t} \leq \lambda_{1}(X)^{t} I.
\end{displaymath}
This inequality implies that $\lambda_{1}(X) \leq \lambda_{1}(X)^{t}$, which is a contradiction because $\lambda_{1}(X) > 1$ and $0 < t < 1$. So $\lambda_{1}(X) \leq 1$, equivalently $X \leq I$.

In order to prove (2), suppose that $\lambda_{m}(X) < 1$. Since $X \geq \lambda_{m}(X) I$, the similar argument as above yields $\lambda_{m}(X) \geq \lambda_{m}(X)^{t}$, but it is a contradiction. Thus, $\lambda_{m}(X) \geq 1$, equivalently $X \geq I$.
\end{proof}

In the following we denote as $\lambda_{M} := \max \{ \lambda_{1}(A_{j}): 1 \leq j \leq n \}$.
\begin{corollary}
Let $0 < t \leq z < 1$. Then for all $p$ such that $p \leq z$
\begin{displaymath}
\mathcal{R}_{t,z}(\omega; \mathbb{A}) \leq \lambda_{M}^{t} \mathcal{Q}_{\frac{1}{p}}(\omega; \mathbb{A}^{1-t}).
\end{displaymath}
\end{corollary}

\begin{proof}
Since $\lambda_{M}^{-1} A_{j} \leq I$ for all $j$, we have $\mathcal{R}_{t,z} \left( \omega; \lambda_{M}^{-1} \mathbb{A} \right) \leq I$ by Lemma \ref{L:supplement} (1). From Theorem \ref{T:R-Q} together with the homogeneity of the R\'{e}nyi power mean,
\begin{displaymath}
\lambda_{M}^{-1} \mathcal{R}_{t,z}(\omega; \mathbb{A}) = \mathcal{R}_{t,z} \left( \omega; \lambda_{M}^{-1} \mathbb{A} \right) \leq \left( \sum_{j=1}^{n} w_{j} (\lambda_{M}^{-1} A_{j})^{\frac{1-t}{p}} \right)^{p} = \lambda_{M}^{t-1} \left( \sum_{j=1}^{n} w_{j} A_{j}^{\frac{1-t}{p}} \right)^{p}.
\end{displaymath}
By simplifying the terms of $\lambda_{M}$ we complete the proof.
\end{proof}

\begin{theorem}
Let $0 < t \leq z < 1$. Then
\begin{displaymath}
\mathcal{R}_{t,z}(\omega; \mathbb{A})^{\frac{1-t}{2}} \geq \lambda_{M}^{-\frac{(1-t)(1-z)}{2z}} \sum_{j=1}^{n} w_{j} A_{j}^{\frac{1-t}{2z}}.
\end{displaymath}
\end{theorem}

\begin{proof}
We first assume that $A_{j} \leq I$ for all $j$. Let $X = \mathcal{R}_{t,z}(\omega; \mathbb{A})$ for $0 < t \leq z < 1$.
By \eqref{E:Hansen2}
\begin{displaymath}
X = \sum_{j=1}^{n} w_{j} \left( A_{j}^{\frac{1-t}{2z}} X^{\frac{t}{z}} A_{j}^{\frac{1-t}{2z}} \right)^{z} \geq \sum_{j=1}^{n} w_{j} A_{j}^{\frac{1-t}{2z}} X^{t} A_{j}^{\frac{1-t}{2z}}.
\end{displaymath}
Taking the congruence transformation by $X^{\frac{t}{2}}$ and applying the convexity of a square map yield
\begin{displaymath}
X^{1+t} \geq \sum_{j=1}^{n} w_{j} \left( X^{\frac{t}{2}} A_{j}^{\frac{1-t}{2z}} X^{\frac{t}{2}} \right)^{2} \geq \left( \sum_{j=1}^{n} w_{j} X^{\frac{t}{2}} A_{j}^{\frac{1-t}{2z}} X^{\frac{t}{2}} \right)^{2}.
\end{displaymath}
Since the square root map is operator monotone, we have $\displaystyle X^{\frac{1+t}{2}} \geq \sum_{j=1}^{n} w_{j} X^{\frac{t}{2}} A_{j}^{\frac{1-t}{2z}} X^{\frac{t}{2}}$.
Taking the congruence transformation by $X^{-t/2}$ we obtain
\begin{equation} \label{E:claim1}
X^{\frac{1-t}{2}} \geq \sum_{j=1}^{n} w_{j} A_{j}^{\frac{1-t}{2z}}.
\end{equation}

Now, replacing $A_{j}$ by $\lambda_{M}^{-1} A_{j} (\leq I)$ for all $j$ in \eqref{E:claim1} we have
$$ \displaystyle \mathcal{R}_{t,z} \left( \omega; \lambda_{M}^{-1} \mathbb{A} \right)^{\frac{1-t}{2}} \geq \sum_{j=1}^{n} w_{j} \left( \lambda_{M}^{-1} A_{j} \right)^{\frac{1-t}{2z}}. $$
Since the R\'{e}nyi power mean $\mathcal{R}_{t,z}$ is homogeneous, it reduces to
\begin{displaymath}
\lambda_{M}^{\frac{t-1}{2}} \mathcal{R}_{t,z} \left( \omega; \mathbb{A} \right)^{\frac{1-t}{2}} \geq \lambda_{M}^{\frac{t-1}{2z}} \sum_{j=1}^{n} w_{j} A_{j}^{\frac{1-t}{2z}}.
\end{displaymath}
By simplifying the terms of $\lambda_{M}$ we obtain the desired inequality.
\end{proof}

\begin{remark}
The multi-variable matrix mean on the open convex cone $\mathbb{P}_{m}$ can be defined as a map $G: \Delta_{n} \times \mathbb{P}_{m}^{n} \to \mathbb{P}_{m}$ satisfying the idempotency: $G(\omega; A, \dots, A) = A$ for any $\omega \in \Delta_{n}$ and $A \in \mathbb{P}_{m}$.
Boundedness of the multi-variable matrix mean plays an important role in operator inequality and majorization. Especially, the multi-variable matrix mean $G$ satisfying the arithmetic-$G$-harmonic mean inequalities
\begin{displaymath}
\left( \sum_{j=1}^{n} w_{j} A_{j}^{-1} \right)^{-1} \leq G(\omega; A_{1}, \dots, A_{n}) \leq \sum_{j=1}^{n} w_{j} A_{j}
\end{displaymath}
fulfills the extended version of Lie-Trotter formula \cite{HK17}:
\begin{equation} \label{E:Lie-Trotter}
\lim_{s \to 0} G(\omega; A_{1}^{s}, \dots, A_{n}^{s})^{1/s} = \exp \left( \sum_{j=1}^{n} w_{j} \log A_{j} \right).
\end{equation}
See \cite{HK19, Kim} for more information. We here have established boundedness of the R\'{e}nyi power mean, but it is still open whether \eqref{E:Lie-Trotter} holds for the R\'{e}nyi power mean.
\end{remark}

\vspace{1cm}
\textbf{Acknowledgement}

No potential competing interest was reported by authors. The work of S. Kim was supported by the National Research Foundation of Korea (NRF) grant funded by the Korea government (MSIT) (No. NRF-2022R1A2C4001306).



\begin{thebibliography}{99}

\bibitem{AC}
M. Agueh and G. Carlier, \textit{Barycenters in the Wasserstein space}, SIAM J. Math. Anal. Appl. \textbf{43} (2011), 904-924.


\bibitem{ABCM}
P. C. \'{A}lvarez-Esteban, E. del Barrio, J. A. Cuesta-Albertos, and C. Mat\'{r}an, \textit{A fixed point approach to barycenters in Wasserstein space}, J. Math. Anal. Appl. \textbf{441} (2016), 744-762.

\bibitem{An94}
T. Ando, \textit{Majorizations and inequalities in matrix theory}, Linear Algbera Appl. \textbf{199} (1994), 17-67.

\bibitem{AH}
T. Ando and F. Hiai, \textit{Log majorization and complementary Golden-Thompson type inequalities}, Linear Algebra Appl. \textbf{197/198} (1994), 113-131.

\bibitem{AD}
K. Audenaert and N. Datta, \textit{$\alpha$-z-Renyi relative entropies}, J. Math. Phys. \textbf{56} (2015), 022202.

\bibitem{Bh97}
R. Bhatia, \textit{Matrix Analysis}, Vol. 169 of Graduate Texts in Mathematics, Springer-Verlag, New York, 1997.

\bibitem{BLY}
R. Bhatia, Y. Lim, and T. Yamazaki, \textit{Some norm inequalities for matrix means}, Linear Algebra Appl. \textbf{501} (2016), 112-122.

\bibitem{BJL19}
R. Bhatia, T. Jain, and Y. Lim, \textit{Inequalities for the Wasserstein mean of positive definite matrices}, Linear Algebra. Appl. \textbf{576} (2019), 108-123.

\bibitem{DDF}
T.-H. Dinh, R. Dumitru, and J. A. Franco, \textit{On a conjecture of Bhatia, Lim and Yamazaki}, Linear Algebra Appl. \textbf{532} (2017), 140-145.

\bibitem{DLVV}
T.-H. Dinh, C.-T. Le, B.-K. Vo and T.-D. Vuong, \textit{The $\alpha-z$-Bures Wasserstein divergence}, Linear Algebra Appl. \textbf{624} (2021), 267-280.

\bibitem{DLV}
T.-H. Dinh, C.-T. Le and T.-D. Vuong, \textit{$\alpha-z$-fidelity and $\alpha-z$-weighted right mean}, submitted (2023), DOI: 10.13140/RG.2.2.24952.72960.

\bibitem{DF}
R. Dumitru and J. A. Franco, \textit{The R\'{e}nyi power means of matrices}, Linear Algebra Appl. \textbf{607} (2020), 45-57.





\bibitem{HP}
F. Hansen and G. K. Pedersen, \textit{Jensen’s inequality for operators and Loewner’s theorem}, Math. Ann. \textbf{258} (1982), 229–241.

\bibitem{Hi}
F. Hiai, \textit{Log-majorization and norm inequalities for exponential operators}. Linear operators (Warsaw, 1994), 119-181, Banach Center Publ. \textbf{38}, Polish Acad. Sci. Inst. Math., Warsaw, 1997.




\bibitem{HJK21}
J. Hwang, M. Jeong and S. Kim, \textit{Right R\'{e}nyi mean and tensor product}, J. Appl. Math. Informatics, Vol. \textbf{39} (2021), No. 5-6, 751-760.

\bibitem{HK17}
J. Hwang and S. Kim, \textit{Lie-Trotter means of positive definite operators}, Linear Algbera Appl. \textbf{531} (2017), 268-280.

\bibitem{HK19}
J. Hwang and S. Kim, \textit{Bounds for the Wasserstein mean with applications to the Lie-Trotter mean}, J. Math. Anal. Appl. \textbf{475} (2019), 1744-1753.

\bibitem{JHK}
M. Jeong, J. Hwang and S. Kim, \textit{Right mean for the $\alpha-z$ Bures-Wasserstein quantum divergence}, Acta Math. Sci., to be published (2023).


\bibitem{Kim18}
S. Kim, \textit{The quasi-arithmetic means and Cartan barycenters of compactly supported measures}, Forum Math. \textbf{30} (2018), no. 3, 753-765.

\bibitem{Kim}
S. Kim, \textit{Parameterized Wasserstein means}, J. Math. Anal. Appl. \textbf{525} (2023), 127272.

\bibitem{LL14}
J. Lawson and Y. Lim, \textit{Karcher means and Karcher equations of positive definite operators}, Trans. Amer. Math. Soc. Series B \textbf{1} (2014), 1-22.


\bibitem{LP12}
Y. Lim and M. P\'{a}lfia, \textit{Matrix power means and the Karcher mean}, J. Funct. Anal. \textbf{262} (2012), 1498-1514.


\bibitem{LY13}
Y. Lim and T. Yamazaki, \textit{On some inequalities for the matrix power and Karcher means}, Linear Algebra Appl. \textbf{438} (2013), 1293-1304.

\bibitem{MO}
M. Mosonyi and T. Ogawa, \textit{Divergence radii and the strong converse exponent of classical-quantum channel coding with constant compositions}, IEEE Trans. Information Theory, Vol. \textbf{67}, No. 3 (2021), 1668-1698.


\bibitem{WWY}
M. Wilde, A. Winter and D. Yang, \textit{Strong converse for the classical capacity of entanglement-breaking and Hadamard channels via a sandwiched Renyi relative entropy}, Comm. Math. Phys. \textbf{331} (2014), 593-622.

\bibitem{Ya12}
T. Yamazaki, \textit{The Riemannian mean and matrix inequalities related to the Ando-Hiai inequality and chaotic order}, Oper. Matrices \textbf{6} (2012), 577-588.

\end{thebibliography}
\end{document}